\documentclass[12pt,reqno]{amsart}
\usepackage{amsmath}
\usepackage{amssymb}
\usepackage[left=2.7cm,top=2.7cm,right=2.7cm,bottom=2.7cm]{geometry}
\usepackage{epsfig}

\begin{document}
\newtheorem{theorem}{Theorem}[section]
\newtheorem{lemma}[theorem]{Lemma}
\newtheorem{definition}[theorem]{Definition}
\newtheorem{conjecture}[theorem]{Conjecture}
\newtheorem{proposition}[theorem]{Proposition}
\newtheorem{claim}[theorem]{Claim}
\newtheorem{algorithm}[theorem]{Algorithm}
\newtheorem{corollary}[theorem]{Corollary}
\newtheorem{observation}[theorem]{Observation}
\newtheorem{problem}[theorem]{Open Problem}
\newcommand{\noin}{\noindent}
\newcommand{\ind}{\indent}
\newcommand{\om}{\omega}
\newcommand{\pp}{\mathcal P}
\newcommand{\AC}{\mathcal A \mathcal C}
\newcommand{\bAC}{\overline{\AC}}
\newcommand{\ppp}{\mathfrak P}
\newcommand{\N}{{\mathbb N}}
\newcommand{\LL}{\mathbb{L}}
\newcommand{\R}{{\mathbb R}}
\newcommand{\E}{\mathbb E}
\newcommand{\Prob}{\mathbb{P}}
\newcommand{\eps}{\varepsilon}

\newcommand{\Ss}{{\mathcal S}}
\newcommand{\Nn}{{\mathcal N}}

\newcommand{\ceil}[1]{\left \lceil #1 \right \rceil}
\newcommand{\floor}[1]{\left \lfloor #1 \right \rfloor}
\newcommand{\size}[1]{\left \vert #1 \right \vert}
\newcommand{\dist}{\mathrm{dist}}

\title{An alternative proof of the linearity of the size-Ramsey number of paths}

\author{Andrzej Dudek}
\address{Department of Mathematics, Western Michigan University, Kalamazoo, MI, USA}
\email{\tt andrzej.dudek@wmich.edu}
\thanks{The first author is supported in part by Simons Foundation Grant \#244712 and by a grant from the Faculty Research and Creative Activities Award (FRACAA), Western Michigan University.}

\author{Pawe\l{} Pra\l{}at}
\address{Department of Mathematics, Ryerson University, Toronto, ON, Canada}
\email{\tt pralat@ryerson.ca}
\thanks{The second author is supported in part by NSERC and Ryerson University}

\thanks{Work done during a visit to the Institut Mittag-Leffler (Djursholm, Sweden)}


\begin{abstract}
The size-Ramsey number $\hat{r}(F)$ of a graph $F$ is the smallest integer $m$ such that there exists a graph $G$ on $m$ edges with the property that every colouring of the edges of $G$ with two colours yields a monochromatic copy of $F$. In 1983, Beck provided a beautiful argument that shows that $\hat{r}(P_n)$ is linear, solving a problem of Erd\H{o}s. In this note, we provide another proof of this fact that actually gives a better bound, namely, $\hat{r}(P_n) < 137n$ for $n$ sufficiently large. 
\end{abstract}

\maketitle

\section{Introduction}\label{sec:intro}

For given two finite graphs $F$ and $G$, we write $G \to F$ if for \emph{every} colouring of the edges of $G$ with two colours (say blue and red) we obtain a monochromatic copy of $F$ (that is, a copy that is either blue or red). The \emph{size-Ramsey number} of a graph $F$, introduced by Erd\H{o}s, Faudree, Rousseau and Schelp~\cite{ErdFauRouSch78} in 1978,  
is defined as follows:
$$
\hat{r}(F) = \min \{ |E(G)| : G \to F \}.
$$

In this note,  we consider the size-Ramsey number of the path $P_n$ on $n$ vertices. It is obvious that $\hat{r}(P_n) = \Omega(n)$ and that $\hat{r}(P_n) = O(n^2)$ (for example, $K_{2n}\to P_n$) but the exact behaviour of $\hat{r}(P_n)$ was not known for a long time. In fact, Erd\H{o}s~\cite{Erd} offered \$100 for a proof or disproof that 
$$
\hat{r}(P_n) / n \to \infty \quad \text{ and } \quad \hat{r}(P_n) / n^2 \to 0.
$$
The problem was solved by Beck~\cite{Beck} in 1983 who, quite surprisingly, showed that $\hat{r}(P_n) < 900 n$ for sufficiently large $n$. A variant of his proof was provided by Bollob\'{a}s~\cite{bol} and it gives $\hat{r}(P_n) < 720 n$ for sufficiently large $n$. It is worth mentioning that both of these bounds are not explicit constructions. Later Alon and Chung~\cite{AlonChung} gave an explicit construction of graphs $G$ on $O(n)$ vertices with $G\to P_n$.

Here we provide an alternative and elementary proof of the linearity of the size-Ramsey number of paths that gives a better bound. The proof relies on a simple observation, Lemma~\ref{lem:obs}, which may be applicable elsewhere.

\begin{theorem}\label{thm:main}
For $n$ sufficiently large, $\hat{r}(P_n) < 137n$.
\end{theorem}

In order to show the result, similarly to Beck and Bollob\'{a}s, we are going to use binomial random graphs. The \emph{binomial random graph} $G(n,p)$ is the random graph $G$ on vertex set $[n]$ for which for every pair $\{i,j\} \in \binom{[n]}{2}$, $\{i,j\}$ appears independently as an edge in $G$ with probability~$p$. Note that $p=p(n)$ may, and usually does, tend to zero as $n$ tends to infinity. All asymptotics throughout are as $n \rightarrow \infty $. We say that a sequence of events $\mathcal{E}_n$ in a probability space holds \emph{asymptotically almost surely} (or \emph{a.a.s.}) if the probability that $\mathcal{E}_n$ holds tends to $1$ as $n$ goes to infinity. 
For simplicity, we do not round numbers that are supposed to be integers either up or down; this is justified since
these rounding errors are negligible to the asymptomatic calculations we will make.

\section{Proof of Theorem~\ref{thm:main}}

We start with the following elementary observation.\footnote{A similar result was independently obtained by Pokrovskiy~\cite{P}.}

\begin{lemma}\label{lem:obs}
Let $c>1$ be a real number and let $G = (V,E)$ be a graph on $cn$ vertices. Suppose that the edges of $G$ are coloured with the colours blue and red and there is no monochromatic~$P_n$. Then the following two properties hold:
\begin{enumerate}
\item [(i)] there exist two disjoint sets $U,W \subseteq V$ of size $n(c-1)/2$ such that there is no blue edge between $U$ and $W$,
\item [(ii)] there exist two disjoint sets $U',W' \subseteq V$ of size $n(c-1)/2$ such that there is no red edge between $U'$ and $W'$.
\end{enumerate}
\end{lemma}
\begin{proof}
We perform the following algorithm on $G$ and construct a blue path $P$. Let $v_1$ be an arbitrary vertex of $G$, let $P=(v_1)$, $U = V \setminus \{v_1\}$, and $W = \emptyset$. We investigate all edges from $v_1$ to $U$ searching for a blue edge. If such an edge is found (say from $v_1$ to $v_2$), we extend the blue path as $P=(v_1,v_2)$ and remove $v_2$ from $U$. We continue extending the blue path $P$ this way for as long as possible. Since there is no monochromatic $P_n$, we must reach the point of the process in which $P$ cannot be extended, that is, there is a blue path from $v_1$ to $v_k$ ($k < n$) and there is no blue edge from $v_k$ to $U$. This time, $v_k$ is moved to $W$ and we try to continue extending the path from $v_{k-1}$, reaching another critical point in which another vertex will be moved to $W$, etc. If $P$ is reduced to a single vertex $v_1$ and no blue edge to $U$ is found, we move $v_1$ to $W$ and simply re-start the process from another vertex from $U$, again arbitrarily chosen. 

An obvious but important observation is that during this algorithm there is never a blue edge between $U$ and $W$. Moreover, in each step of the process, the size of $U$ decreases by 1 or the size of $W$ increases by 1. Finally, since there is no monochromatic $P_n$, the number of vertices of the blue path $P$ is always smaller than $n$. Hence, at some point of the process both $U$ and $W$ must have size at least $n(c-1)/2$. Part (i) now holds after removing some vertices from $U$ or $W$, if needed, so that both sets have sizes precisely $n(c-1)/2$. 

Part (ii) can be proved by a symmetric argument; this time the algorithm tries to build a red path. The proof is finished.
\end{proof}

Now, we prove the following straightforward properties of random graphs. For every two disjoint sets $S$ and $T$, $e(S,T)$ denotes the number of edges between $S$ and $T$.

\begin{lemma}\label{lem:random1}
Let $c=7.29$ and $d=5.14$, and consider $G = (V,E) \in G(cn,d/n)$. Then, the following two properties hold a.a.s.:
\begin{enumerate}
\item [(i)] $|E(G)| = (1+o(1)) n c^2 d / 2 < 137 n$,
\item [(ii)] for every two disjoint sets of vertices $S$ and $T$ such that $|S| = |T| = n(c-3)/4$ we have $e(S,T) \neq 0$.
\end{enumerate}
\end{lemma}
\begin{proof}
Part (i) is obvious. The expected number of edges in $G$ is ${cn \choose 2} \frac {d}{n} = (1+o(1)) n c^2 d / 2$, and the concentration around the expectation follows immediately from Chernoff's bound. 

For part (ii), let $X$ be the number of pairs of disjoint sets $S$ and $T$ of desired size such that $e(S,T) = 0$. Putting $\alpha = \alpha(c) = (c-3)/4$ for simplicity, we get
\begin{eqnarray*}
\E[X] &=& {cn \choose \alpha n} {(c- \alpha)n \choose \alpha n} \left( 1- \frac{d}{n} \right)^{\alpha n\cdot \alpha n} \\
& \le & \frac {(cn)!}{(\alpha n)! (\alpha n)! ((c-2\alpha)n)!} \exp \big( -d \alpha^2 n \big).
\end{eqnarray*}
Using Stirling's formula ($x! = (1+o(1)) \sqrt{2\pi x} (x/e)^x$) we get that $\E[X] \le \exp( f(c,d) n)$, where
$$
f(c,d) = c \ln c - 2 \alpha \ln \alpha - (c-2\alpha) \ln (c-2\alpha) - d \alpha^2.
$$
Putting numerical values of $c$ and $d$ into the formula, we get $f(c,d) < -0.008$ and so $\E[X] \to 0$ as $n \to \infty$. 
(The values of $c$ and $d$ were chosen so as to minimize $c^2 d / 2$ under the condition $f(c,d) < 0$.) Now part (ii) holds by Markov's inequality. 
\end{proof}

Now, we are ready to prove the main result.

\begin{proof}[Proof of Theorem~\ref{thm:main}]
Let $c=7.29$ and $d=5.14$, and consider $G = (V,E) \in G(cn,d/n)$. We show that a.a.s.\ $G \to P_n$ which will finish the proof by Lemma~\ref{lem:random1}(i).

For a contradiction, suppose that $G \not\to P_n$. Thus, there is a blue-red colouring of $E$ with no monochromatic~$P_n$. It follows (deterministically) from Lemma~\ref{lem:obs}(i) that $V$ can be partitioned into three sets $P, U, W$ such that $|P|=n$, $|U|=|W|=n(c-1)/2$, and there is no blue edge between $U$ and $W$. Similarly, by Lemma~\ref{lem:obs}(ii), $V$ can be partitioned into three sets $P', U', W'$ such that $|P'|=n$, $|U'|=|W'|=n(c-1)/2$, and there is no red edge between $U'$ and $W'$.

Now, consider $X = U \cap U', Y = U \cap W', X' = W \cap U', Y' = W \cap W'$
and let $x =|X|, y=|Y|, x'=|X'|, y'=|Y'|$ be their sizes, respectively. Observe that
\begin{equation}\label{eq:x+y}
x+y = |U \cap (U' \cup W')| = |U \setminus P'| \ge |U| - |P'| = n (c-3)/2.
\end{equation}
Similarly, one can show that $x'+y' \ge n (c-3)/2$, $x+x' \ge n (c-3)/2$, and that $y+y' \ge n (c-3)/2$. We say that a set is \emph{large} if its size is at least $n(c-3)/4$; otherwise, we say that it is \emph{small}. We need the following straightforward observation.

\bigskip

\noindent \emph{Claim.} Either both $X$ and $Y'$ are large or both $Y$ and $X'$ are large.

\smallskip

\noindent (In fact one can easily show that the constant $(c-3)/4$ in the definition of being large is optimal.)

\smallskip

\noindent \emph{Proof of the claim.} For a contradiction, suppose that at least one of $X,Y'$ is small \emph{and} at least one of $Y,X'$ is small, say, $X$ and $Y$ are small. But this implies that $x+y < n(c-3)/4 + n(c-3)/4 = n(c-3)/2$, which contradicts~(\ref{eq:x+y}). The remaining three cases are symmetric, and so the claim holds. 

\bigskip

Now, let us come back to the proof. Without loss of generality, we may assume that $X=U \cap U'$ and $Y'=W \cap W'$ are large. Since $X \subseteq U$ and $Y' \subseteq W$, there is no blue edge between $X$ and $Y'$. Similarly, one can argue that there is no red edge between $X$ and $Y'$, and so $e(X,Y')=0$. On the other hand, Lemma~\ref{lem:random1}(ii) implies that a.a.s.\ $e(X,Y') \neq 0$, reaching the desired contradiction. It follows that a.a.s.\ $G \to P_n$ which finishes the proof.
\end{proof}

\section{Remarks}

In this note we showed that $\hat{r}(P_n) < 137n$. On the other hand, the best known lower bound, $\hat{r}(P_n) \ge (1+\sqrt{2})n-2$, was given by Bollob\'as~\cite{bol2} who improved the previous result of Beck~\cite{Beck2} that shows that $\hat{r}(P_n) \ge \frac{9}{4}n$. Decreasing the gap between the lower and upper bounds might be of some interest. 
One approach to improving the upper bound could be to deal with non-symmetric cases in our claim or to use random $d$-regular graphs instead of binomial graphs.

Another related problem deals with longest monochromatic paths in $G(n,p)$. Observe that it follows from the proof of Theorem~\ref{thm:main} that for every $\omega = \omega(n)$ tending to infinity arbitrarily slowly together with $n$ we have that a.a.s.\ any 2-colouring of the edges of $G(n,\omega/n)$ yields a monochromatic path of length $\frac{(1-\varepsilon)}{3}n$ for an arbitrarily small $\varepsilon>0$. On the other hand, a simple construction of Gerencs{\'e}r and Gy{\'a}rf{\'a}s~\cite{Gyarfas} shows that such path cannot be longer than $\frac{2}{3}n$. We conjecture that actually $(1+o(1))\frac{2}{3}n$ is the right answer for random graphs with average degree tending to infinity.\footnote{The conjecture was recently proved by Letzter~\cite{L}.}

\section{Acknowledgment}
We would like to thank to the referees and editors for their valuable comments and suggestions.

\end{document}